\documentclass[11pt,oneside,reqno]{amsart}
\usepackage[margin=1.2in]{geometry}
\usepackage{amsmath,amssymb,amsthm,mathtools,xurl}
\usepackage{microtype}
\usepackage{enumitem}
\usepackage[colorlinks=true,linkcolor=blue,citecolor=blue,urlcolor=blue]{hyperref}

\newtheorem{theorem}{Theorem}[section]
\newtheorem{proposition}[theorem]{Proposition}
\newtheorem{lemma}[theorem]{Lemma}
\newtheorem{corollary}[theorem]{Corollary}
\theoremstyle{remark}
\newtheorem{remark}[theorem]{Remark}

\numberwithin{equation}{section}

\newcommand{\R}{\mathbb R}

\title[Blow-up Inhomogeneous BNLS]{Nonradial Finite-Time Blow-up for the Mass-Critical Inhomogeneous Biharmonic NLS}

\author[N.~U.~C.~Nguyen, V.~D.~Dinh]{Ngoc Uyen Cong Nguyen and Van Duong Dinh}

\address[N.~U.~C. Nguyen]{Claire and Leo LLC, Houston, Texas, USA}
\email{contact@uyencong.com}

\address[V.~D.~Dinh]{Claire and Leo LLC, Houston, Texas, USA}
\email{contact@duongdinh.com}
    
\date{September 19, 2026}



\begin{document}
\maketitle

\begin{abstract}
We consider the focusing mass-critical inhomogeneous biharmonic nonlinear Schr\"odinger equation
\[
i\partial_tu=\Delta^2u-|x|^{-b}|u|^{\frac{8-2b}{N}}u, \quad 0<b<\min\{4,N/2\}.
\]
We prove that every \(H^2(\mathbb R^N)\) solution with negative energy blows up in finite time, both forward and backward. No radial symmetry or finite-variance assumption is required. In particular, this rules out the global grow-up scenario in the previously known negative-energy blow-up alternative and establishes finite-time blow-up in the genuinely nonradial setting. The proof relies on a new localized virial argument exploiting the spatial inhomogeneity to control the localization errors.

\end{abstract}

\section{Introduction and main result}
\label{sec:intro}

We consider the focusing mass-critical inhomogeneous biharmonic nonlinear Schr\"odinger equation
\begin{equation}\label{eq:ibnls}
\begin{cases}
i\partial_tu=\Delta^2u-|x|^{-b}|u|^{\frac{8-2b}{N}} u,\\
u(0)=u_0\in H^2(\R^N),
\end{cases}
\quad 0<b<\min\{4,N/2\}.
\end{equation}
Equation \eqref{eq:ibnls} is $L^2$-critical (mass-critical) since it is invariant under the scaling
\[
u_\lambda(t,x)=\lambda^{N/2}u(\lambda^4t,\lambda x),
\]
which also preserves the $L^2$-norm. Along the maximal $H^2$ solution, the mass
$$
M(u)=\|u(t)\|_{L^2}^2
$$
and the energy
$$
E(u)=\frac12\|\Delta u(t)\|_{L^2}^2 -\frac{N}{2N+8-2b}\int_{\R^N}|x|^{-b}|u(t,x)|^{\frac{2N+8-2b}{N}} dx
$$
are conserved.

Fourth-order nonlinear Schr\"odinger equations were introduced by Karpman \cite{Karpman} and Karpman--Shagalov \cite{KS} as higher-order dispersive models arising in nonlinear optics. Since then, biharmonic Schr\"odinger equations have attracted considerable attention as natural higher-order counterparts of the classical nonlinear Schr\"odinger equation; see, among many others, \cite{BKS,BL,FIP,Pausader}. The inhomogeneous factor $|x|^{-b}$ introduces a spatially decaying nonlinear interaction and leads to a substantially different dynamics from the translation-invariant problem. Local and global well-posedness, stability, and scattering for inhomogeneous biharmonic Schr\"odinger equations have been studied in \cite{Campos-Guzman,Cardoso-Guzman-Pastor,DK, Guzman-Pastor,Liu-Zhang,Saanouni} and the references therein.

Our concern here is finite-time blow-up at the mass-critical power. For the homogeneous biharmonic NLS, a fundamental contribution is due to Boulenger and Lenzmann \cite{BL}, who developed localized virial and Riesz-bivariance methods adapted to the fourth-order flow and established a general blow-up theory in the mass-critical and mass-supercritical regimes. In contrast with the classical NLS, however, the biharmonic equation does not possess a variance identity whose concavity directly forces finite-time blow-up. In particular, for the pure mass-critical biharmonic flow, their argument yields for radial negative-energy solutions only the alternative between finite-time blow-up and global grow-up.

Finite-time blow-up for inhomogeneous fourth-order Schr\"odinger equations was first investigated by Cho, Ozawa and Wang \cite{COW}, who considered the special mass-critical nonlinearity $|x|^{-2}|u|^{\frac{4}{N}} u$. Under suitable dimensional restrictions, they proved finite-time blow-up for negative-energy solutions assuming either radial symmetry or finite variance. For the same $b=2$ mass-critical model, the second author \cite{Dinh-AML} subsequently exploited the spatial decay of the inhomogeneous nonlinearity to remove these assumptions and proved finite-time blow-up for nonradial negative-energy solutions in the presence of a nontrivial lower-order dispersion term. In the intercritical regime, Saanouni and Ghanmi \cite{Saanouni-Ghanmi} established blow-up criteria and strong instability of standing waves, while, more recently, Scarpelli and Hespanha \cite{Scarpelli-Hespanha} obtained radial and nonradial blow-up results for the energy-critical inhomogeneous biharmonic equation.

The result most closely related to the present work is due to Bai, Majdoub and Saanouni \cite{Bai-Majdoub-Saanouni}, who considered the general mass-critical equation
\[ 
i\partial_tu-\Delta^2u+\nu\Delta u =-|x|^{-b}|u|^{\frac{8-2b}{N}}u,
 \quad
 0<b<\min\{4,N/2\},\quad \nu\ge0.
\]
When $\nu>0$, they proved finite-time blow-up for arbitrary negative-energy $H^2$ solutions, without radial symmetry. For the pure biharmonic case $\nu=0$, however, their argument yields only the alternative
\[
T_+<\infty
\quad\text{or}\quad
\|\Delta u(t)\|_{L^2}\gtrsim t^2
\quad\text{for }t\gg1.
\]
Thus, even in the presence of the spatial inhomogeneity, the possibility of a global negative-energy solution undergoing unbounded grow-up remained open.

This phenomenon has no analogue in the corresponding mass-critical inhomogeneous second-order NLS. Cardoso and Farah \cite{Cardoso-Farah} proved that negative energy forces finite-time blow-up for arbitrary $H^1$ solutions, without radial symmetry or finite variance. It is therefore natural to ask whether the inhomogeneity is likewise strong enough to eliminate the grow-up branch for the pure biharmonic flow. The main result of this paper answers this question affirmatively: we rule out the global grow-up scenario in \cite{Bai-Majdoub-Saanouni} and prove that every negative-energy $H^2$ solution of the mass-critical inhomogeneous pure biharmonic NLS blows up in finite time, without any radial or finite-variance assumption.

We recall the local theory needed to formulate the result. By Liu and Zhang \cite{Liu-Zhang}, for every $u_0\in H^2(\mathbb R^N)$ and
\[
0<b<\min\{4,N/2\},
\]
equation \eqref{eq:ibnls} admits a unique maximal solution
\[
u\in C((-T_-,T_+);H^2(\mathbb R^N)),
\quad
T_-,T_+\in(0,\infty].
\]
The mass and energy are conserved on the maximal lifespan, and the usual blow-up alternative holds:
\[
T_+<\infty
\quad\Longrightarrow\quad
\lim_{t\uparrow T_+}\|u(t)\|_{H^2}=\infty,
\]
with the analogous statement backward in time. 

We are now ready to state our main result.

\begin{theorem}[Nonradial finite-time blow-up]\label{thm:main}
Let $N\ge1$ and $0<b<\min\{4,N/2\}$. Let $u_0\in H^2(\R^N)$, and let
\[
u\in C((-T_-,T_+);H^2(\R^N))
\]
be the maximal solution of \eqref{eq:ibnls}. If $E(u_0)<0$, then
\[
T_+<\infty
\quad\text{and}\quad
T_-<\infty.
\]
In particular,
\[
\lim_{t\uparrow T_+}\|u(t)\|_{H^2}=\infty,
\quad
\lim_{t\downarrow-T_-}\|u(t)\|_{H^2}=\infty.
\]
\end{theorem}

\begin{remark}[On the range of $b$]
The restriction
\[
0<b<\min\{4,N/2\}
\]
has two different origins. The assumption $b>0$ gives the strict inequality $1-\frac{b}{4}<1$, while $b<4$ ensures $1-\frac{b}{4}>0$. The exponent \(\theta=1-\frac{b}{4} \in (0,1)\) will play an important role in the weighted interpolation argument (see Proposition \ref{prop:weighted-GN}). On the other hand, the additional condition $b<N/2$ is inherited from the available $H^2$ local well-posedness theory for the inhomogeneous biharmonic equation; see Liu--Zhang \cite{Liu-Zhang}. In particular, the restriction $b<N/2$ is not intrinsic to the blow-up mechanism developed here.
\end{remark}

\begin{remark}[The homogeneous limit]
The argument developed in this paper is genuinely inhomogeneous. Indeed, when $b=0$, equation \eqref{eq:ibnls} reduces to the homogeneous mass-critical biharmonic NLS, while the crucial gain
\[
1-\frac{b}{4}<1
\]
degenerates to equality. Consequently, the absorption mechanism used in our proof no longer closes, and Theorem~\ref{thm:main} does not extend to the homogeneous problem by simply letting $b\downarrow0$.

During the preparation of this manuscript, we learned from Leonard Dinh that he has independently obtained finite-time blow-up for radial negative-energy solutions of the homogeneous mass-critical biharmonic NLS (see \cite{Leonard}). This resolves, in the radial setting, the finite-time blow-up problem left open by Boulenger and Lenzmann \cite{BL}, whose result allowed the alternative of global grow-up. The result in \cite{Leonard} and the present theorem are therefore complementary: his result concerns the homogeneous equation under radial symmetry, whereas our argument exploits the spatial inhomogeneity $b>0$ to obtain finite-time blow-up without radial symmetry or a finite-variance assumption.
\end{remark}

\begin{remark}[Negative energy]
The condition $E(u_0)<0$ is a sufficient blow-up criterion rather than a characterization of the blow-up region. Solutions with nonnegative energy may also blow up, and a finer description of the threshold dynamics would require additional variational information, typically involving the ground state and the associated Pohozaev functional. The present result isolates the negative-energy regime, in which finite-time blow-up holds for arbitrary $H^2$ data.
\end{remark}

We next briefly describe the main idea behind the proof of Theorem~\ref{thm:main}. The key issue is to strengthen the localized virial mechanism sufficiently to rule out the global grow-up scenario that remains possible in the pure biharmonic case. In the absence of lower-order dispersion, the standard localized virial estimates alone do not force finite-time blow-up. Our approach is inspired by the exponential-weight virial method recently introduced by the authors in \cite{Nguyen-Dinh} for the mass-critical quadratic Schr\"odinger system without mass resonance. In that setting, a bounded exponential localization retains a weighted kinetic defect and ultimately leads to a superlinear Riccati-type inequality. The present fourth-order, nonradial problem requires additional ideas, however, because of the higher-order dispersion and the localization errors generated by the singular inhomogeneous nonlinearity.

To overcome these difficulties, we introduce a localization which is never exactly quadratic away from the origin. More precisely, we choose a radial weight satisfying
$$
\nabla\phi_R(x)=x\exp\left[-(|x|/R)^4\right].
$$
Its Hessian retains a positive defect at every nonzero radius, which both controls the nonlinear localization error and provides the coercivity needed to estimate the localized virial quantity. The spatial inhomogeneity $|x|^{-b}$ enters through a weighted dyadic Gagliardo--Nirenberg estimate with exponent
$$
1-\frac b4<1,
$$
allowing the defect term to absorb the nonlinear error. A further weighted interpolation estimate then yields a superlinear Riccati-type differential inequality, which rules out global existence.

Although the exponential localization and Riccati mechanism are motivated by \cite{Nguyen-Dinh}, the present argument requires a new weighted localization analysis adapted simultaneously to fourth-order dispersion, the singular inhomogeneity, and the absence of radial symmetry.

\medskip

As an application of Theorem~\ref{thm:main}, we also obtain strong instability by finite-time blow-up of standing waves. Strong instability for fourth-order Schr\"odinger equations has been studied in the homogeneous setting by Bonheure, Cast\'eras, Gou and Jeanjean \cite{BCGJ}, and for inhomogeneous fourth-order equations in the intercritical regime by Saanouni and Ghanmi \cite{Saanouni-Ghanmi}. At the mass-critical exponent considered here, the conclusion takes a particularly simple form. Indeed, every nontrivial standing-wave profile has zero energy as a consequence of the Pohozaev identities. An arbitrarily small increase of its amplitude therefore produces negative-energy initial data, to which Theorem~\ref{thm:main} applies.

\begin{corollary}[Strong instability of standing waves]
\label{coro:instability}
Let $N\ge1$ and $0<b<\min\{4,N/2\}$. Let $\omega\in\mathbb R$, and let $0\neq Q\in H^2(\mathbb R^N)$ solve
\begin{equation}\label{eq:stationary}
\Delta^2Q+\omega Q-|x|^{-b}|Q|^{\frac{8-2b}{N}} Q=0.
\end{equation}
Then the standing wave
\[
u(t,x)=e^{i\omega t}Q(x)
\]
is strongly unstable by finite-time blow-up. More precisely, for every $\varepsilon>0$ there exists $u_0\in H^2(\mathbb R^N)$ such that
$$
\|u_0-Q\|_{H^2}<\varepsilon,
$$
and the corresponding solution of \eqref{eq:ibnls} blows up in finite time.
\end{corollary}

\begin{remark}
The conclusion of Corollary~\ref{coro:instability} holds for every nontrivial standing wave. In particular, no ground-state assumption, radial symmetry, or variational minimizing property is required.
\end{remark}

\begin{remark}
Any nontrivial solution $Q$ of \eqref{eq:stationary} necessarily satisfies $\omega>0$. This is an immediate consequence of the Pohozaev identities for the stationary equation; see Lemma~\ref{lem:pohozaev-standing}.
\end{remark}

Throughout the remainder of the paper, we set 
\begin{equation} \label{eq:def-pq}
p:=\frac{8-2b}{N}, \quad q:=p+2=\frac{2N+8-2b}{N}.
\end{equation}
The conserved energy takes the form
\[
E(u)= \frac12\|\Delta u\|_{L^2}^2 -\frac1q\int_{\mathbb R^N}|x|^{-b}|u|^q dx.
\]
We also denote
\[
m:=M(u_0)=\|u(t)\|_{L^2}^2.
\]
The paper is organized as follows. In Section~2, we introduce the exponential multiplier and the associated defect quantities. Section~3 is devoted to the weighted interpolation estimates that control the inhomogeneous nonlinear and lower-order error terms. In Section~4, we derive the coercive localized virial inequality, while Section~5 establishes the virial--defect estimate needed to close the argument. Section~6 contains the proof of the main finite-time blow-up theorem. Finally, in Section~7, we prove the strong instability of standing waves.

\section{The exponential multiplier and its defect}
\label{sec:multiplier}

For $R>0$, set $r=|x|$, $s=r/R$, and
\[
h(s):=e^{-s^4}.
\]
We define the radial function $\phi_R$ by
\[
\phi_R'(r)=rh(r/R), \quad \phi_R(0)=0.
\]
Equivalently,
\[
\nabla\phi_R(x)=xh(|x|/R).
\]
Associated with this multiplier is the defect
\[
\omega_R(x):=1-h(|x|/R)=1-e^{-|x|^4/R^4}.
\]
The elementary estimates
\[
c\min\left\{\frac{|x|^4}{R^4},1\right\}
\le
\omega_R(x)
\le
\min\left\{\frac{|x|^4}{R^4},1\right\}
\]
hold uniformly in $R>0$. We shall repeatedly use the positive quantity
\[
\mathcal D_R(u) := \int_{\mathbb R^N} \omega_R(x)|D^2u(x)|_F^2 dx,
\]
where
\[
|D^2u|_F^2:=\sum_{j,k=1}^N|\partial_{jk}u|^2.
\]
By Plancherel's identity,
\begin{equation} \label{eq:plancherel-iden}
\int_{\mathbb R^N}|D^2u|_F^2 dx=\|\Delta u\|_{L^2}^2.
\end{equation}

We first record the geometry of the multiplier.

\begin{lemma}[Geometry of the exponential multiplier]
\label{lem:multiplier-geometry}
The tangential and radial eigenvalues of $D^2\phi_R$ are respectively
\[
\lambda_{\rm tan}(r)=h(s), \quad \lambda_{\rm rad}(r)=h(s)(1-4s^4).
\]
Consequently,
\[
I-D^2\phi_R\ge\omega_R I
\]
as quadratic forms. Moreover,
\[
\Delta\phi_R=h(s)(N-4s^4).
\]
\end{lemma}

\begin{proof}
For a radial function,
\[
\lambda_{\rm tan}(r)=\frac{\phi_R'(r)}r=h(s).
\]
On the other hand,
\[
\phi_R''(r) =h(s)+s h'(s) =h(s)(1-4s^4),
\]
which gives the radial eigenvalue. Both eigenvalues are bounded above by
$h(s)=1-\omega_R$, and hence
\[
D^2\phi_R\le(1-\omega_R)I.
\]
Therefore
\[
I-D^2\phi_R\ge\omega_RI.
\]
Finally,
\[
\Delta\phi_R =\phi_R''+\frac{N-1}{r}\phi_R' =h(s)(N-4s^4).
\]
\end{proof}

The derivatives of $\phi_R$ enjoy bounds adapted to the same defect.

\begin{lemma}[Derivative bounds]
\label{lem:multiplier-derivatives}
There exists $C_N>0$ such that
\[
|\nabla^2\Delta\phi_R|+|\Delta^2\phi_R|\le C_NR^{-2}\omega_R^{1/2},
\quad
|\Delta^3\phi_R|\le C_NR^{-4}.
\]
\end{lemma}

\begin{proof}
By scaling,
\[
\phi_R(x)=R^2\phi_1(x/R).
\]
Near the origin,
\[
\phi_1'(r)=re^{-r^4}=r-r^5+O(r^9),
\]
and hence
\[
\phi_1(r)=\frac{r^2}{2}-\frac{r^6}{6}+O(r^{10}).
\]
Thus every fourth derivative of $\phi_1-|x|^2/2$ is $O(|x|^2)$ near the
origin, while all sixth derivatives remain bounded. Since
\[
\omega_1^{1/2}(x)\simeq\min\{|x|^2,1\},
\]
and all derivatives of the exponential profile decay rapidly at infinity,
the corresponding estimates hold globally for $R=1$. Scaling gives the
result.
\end{proof}

\section{Weighted interpolation estimates}
\label{sec:weighted-estimates}

This section is devoted to the following weighted Gagliardo--Nirenberg estimate.

\begin{proposition}[Weighted Gagliardo--Nirenberg estimate]
\label{prop:weighted-GN}
Let $N\geq 1, 0<b<\min \{4,N/2\}$ and $p,q$ be as in \eqref{eq:def-pq}. Then, for every $\varepsilon>0$, there exists $C_{N,b,\varepsilon}>0$ such that, for every $u\in H^2(\mathbb R^N)$ and every $R>0$,
\begin{equation} \label{eq:weighted-GN}
\int_{\mathbb R^N} \omega_R(x)|x|^{-b}|u(x)|^q dx \leq \varepsilon \mathcal D_R(u)+C_{N,b,\varepsilon}R^{-4}\left(\|u\|_{L^2}^{2+4p/b}+\|u\|_{L^2}^{q}\right).
\end{equation}
\end{proposition}

\begin{proof}
Set
\begin{equation} \label{eq:choi-theta-gamma}
\theta:=1-\frac b4\in(0,1), \quad \gamma:=\frac q2-\theta=\frac p2+\frac b4.
\end{equation}
Since $q=p+2$ and $Np=8-2b$, we have
\[
\theta=\frac{N(q-2)}8=\frac{Np}{8}.
\]

\medskip
\noindent
\emph{Step 1. Annular localization and control of the localized $H^2$ norm.}

For $j\in\mathbb Z$, set
$$
\rho_j:=2^jR, \quad A_j:=\{x\in\mathbb R^N:\rho_j<|x|<2\rho_j\}.
$$
Choose enlarged annuli
$$
A_j\Subset A_j^*\Subset A_j^{**}
$$
with uniformly finite overlap, and choose
$\eta_j\in C_c^\infty(A_j^*)$ such that
$$
\eta_j\equiv1\quad\text{on }A_j, \quad |\nabla^k\eta_j|\lesssim \rho_j^{-k},
\quad k=1,2.
$$
Define
$$
d_j:=\min\left\{\left(\frac{\rho_j}{R}\right)^4,1\right\},
$$
and
$$
m_j:=\int_{A_j^{**}}|u|^2 dx,
\quad
D_j:=\int_{A_j^{**}}\omega_R(x)|D^2u|_F^2 dx.
$$
Since $|x|\simeq\rho_j$ on $A_j^{**}$ and
$$
\omega_R(x)\simeq \min\left\{\frac{|x|^4}{R^4},1\right\},
$$
we have
$$
\omega_R(x)\simeq d_j, \quad x\in A_j^{**},
$$
with constants independent of $j$ and $R$.

We claim that
\begin{equation}
\label{eq:claim-annuli}
\|D^2(\eta_j u)\|_{L^2}^2 \lesssim d_j^{-1}D_j+\rho_j^{-4}m_j.
\end{equation}
Indeed, the product rule and the bounds on $\eta_j$ give
$$
\|D^2(\eta_j u)\|_{L^2}^2 \lesssim \|D^2u\|_{L^2(A_j^*)}^2 +\rho_j^{-2}\|\nabla u\|_{L^2(A_j^*)}^2+\rho_j^{-4}m_j.
$$
To estimate the middle term, choose $\chi_j\in C_c^\infty(A_j^{**})$ satisfying
$$
\chi_j\equiv1\quad\text{on }A_j^*, \quad |\nabla\chi_j|\lesssim\rho_j^{-1}.
$$
Integration by parts gives
$$
\begin{aligned}
\int \chi_j^2|\nabla u|^2dx
&=-\operatorname{Re}\int \chi_j^2\overline u\Delta udx -2\operatorname{Re}\int \chi_j\overline u\nabla\chi_j\cdot\nabla udx.
\end{aligned}
$$
By Cauchy--Schwarz,
$$
\left|\int\chi_j^2\overline u\Delta udx\right|
\lesssim m_j^{1/2} \|D^2u\|_{L^2(A_j^{**})},
$$
while Young's inequality yields
$$
\left|2\int \chi_j\overline u\nabla\chi_j\cdot\nabla udx\right|
\le \frac12\int\chi_j^2|\nabla u|^2dx
+ C\rho_j^{-2}m_j.
$$
Hence
\begin{equation}
\label{eq:localized-gradient}
\|\nabla u\|_{L^2(A_j^*)}^2
\lesssim
m_j^{1/2} \|D^2u\|_{L^2(A_j^{**})}
+\rho_j^{-2}m_j.
\end{equation}
Multiplying by $\rho_j^{-2}$ and applying Young's inequality once more,
$$
\rho_j^{-2} \|\nabla u\|_{L^2(A_j^*)}^2 \lesssim \|D^2u\|_{L^2(A_j^{**})}^2+\rho_j^{-4}m_j.
$$
Since $\omega_R\simeq d_j$ on $A_j^{**}$,
$$
\|D^2u\|_{L^2(A_j^{**})}^2 \lesssim d_j^{-1}D_j.
$$
This proves \eqref{eq:claim-annuli}.

\medskip
\noindent
\emph{Step 2. Local Gagliardo--Nirenberg estimate on each annulus.}

We recall the following $H^2$ Gagliardo--Nirenberg inequality (see e.g. \cite{Nirenberg}):
\[
\|u\|^q_{L^q} \lesssim \|D^2 u \|^{\frac{N(q-2)}{4}}_{L^2} \|u\|^{q-\frac{N(q-2)}{4}}_{L^2},
\]
which, by \eqref{eq:def-pq} and \eqref{eq:choi-theta-gamma}, gives
\[
\|u\|^q_{L^q} \lesssim \|D^2 u \|^{2\theta}_{L^2} \|u\|^{2\gamma}_{L^2}.
\]
Since $\eta_j\equiv1$ on $A_j$, the above inequality gives
$$
\int_{A_j}|u|^qdx
\le \|\eta_j u\|_{L^q}^q
\lesssim \|D^2(\eta_j u)\|_{L^2}^{2\theta} \|\eta_j u\|_{L^2}^{2\gamma}.
$$
Using \eqref{eq:claim-annuli}, together with
$(a+b)^\theta\lesssim a^\theta+b^\theta$ for $0<\theta<1$, we obtain
$$
\int_{A_j}|u|^qdx \lesssim d_j^{-\theta}D_j^\theta m_j^\gamma +\rho_j^{-4\theta}m_j^{q/2}.
$$
On $A_j$ we also have
$$
\omega_R(x)|x|^{-b} \lesssim d_j\rho_j^{-b}.
$$
Therefore
$$
\int_{A_j} \omega_R|x|^{-b}|u|^qdx \lesssim \rho_j^{-b}d_j^{1-\theta} D_j^\theta m_j^\gamma+d_j\rho_j^{-b-4\theta}m_j^{q/2}.
$$
The two coefficients have the uniform bounds
$$
\rho_j^{-b}d_j^{1-\theta}\lesssim R^{-b},
\quad
d_j\rho_j^{-b-4\theta}\lesssim R^{-4}.
$$
Indeed, since
$$
1-\theta=\frac b4, \quad b+4\theta=4,
$$
if $\rho_j\le R$, then
$d_j=(\rho_j/R)^4$, and hence
$$
\rho_j^{-b}d_j^{1-\theta} =R^{-b},
\quad
d_j\rho_j^{-b-4\theta}=R^{-4}.
$$
If $\rho_j\ge R$, then $d_j=1$, so
$$
\rho_j^{-b}\le R^{-b},
\quad
\rho_j^{-b-4\theta}=\rho_j^{-4}\le R^{-4}.
$$
Consequently,
\begin{equation}
\label{eq:annulus-weighted}
\int_{A_j}
\omega_R|x|^{-b}|u|^qdx
\lesssim
R^{-b}D_j^\theta m_j^\gamma+R^{-4}m_j^{q/2}.
\end{equation}

\medskip
\noindent
\emph{Step 3. Summation over the annuli.}

Summing \eqref{eq:annulus-weighted} over $j$ gives
$$
\int_{\mathbb R^N}\omega_R|x|^{-b}|u|^qdx\lesssim R^{-b}\sum_jD_j^\theta m_j^\gamma+R^{-4}\sum_jm_j^{q/2}.
$$
Because the enlarged annuli have uniformly finite overlap,
$$
\sum_jD_j\lesssim\mathcal D_R(u),
\quad
\sum_jm_j\lesssim\|u\|_{L^2}^2.
$$
By H\"older's inequality,
$$
\sum_jD_j^\theta m_j^\gamma
\le \left(\sum_jD_j\right)^\theta
\left(\sum_jm_j^{\gamma/(1-\theta)}\right)^{1-\theta}.
$$
Since
$$
\frac{\gamma}{1-\theta}=1+\frac{2p}{b}>1,
$$
we have
$$
\sum_jm_j^{\gamma/(1-\theta)}\le \left(\sum_jm_j\right)^{\gamma/(1-\theta)},
$$
and hence
$$
\left(\sum_jm_j^{\gamma/(1-\theta)}\right)^{1-\theta} \lesssim \|u\|_{L^2}^{2\gamma}.
$$
Similarly, since $q/2>1$,
$$
\sum_jm_j^{q/2} \le \left(\sum_jm_j\right)^{q/2}\lesssim \|u\|_{L^2}^{q}.
$$
We therefore obtain
\begin{equation}
\label{eq:weighted-GN-preyoung}
\int_{\mathbb R^N} \omega_R|x|^{-b}|u|^qdx
\lesssim R^{-b}\|u\|_{L^2}^{2\gamma} \mathcal D_R(u)^\theta+ R^{-4}\|u\|_{L^2}^{q}.
\end{equation}

\medskip
\noindent
\emph{Step 4. Absorption of the defect term.}

By Young's inequality, for every $\varepsilon>0$,
$$
R^{-b}\|u\|_{L^2}^{2\gamma} \mathcal D_R(u)^\theta \le \varepsilon\mathcal D_R(u)
+ C_{N,b,\varepsilon} R^{-b/(1-\theta)}\|u\|_{L^2}^{2\gamma/(1-\theta)}.
$$
Finally,
$$
\frac{b}{1-\theta}=4, \quad \frac{2\gamma}{1-\theta}=2+\frac{4p}{b}.
$$
Inserting these identities into \eqref{eq:weighted-GN-preyoung} yields
$$
\int_{\mathbb R^N} \omega_R(x)|x|^{-b}|u(x)|^qdx
\le\varepsilon\mathcal D_R(u)+C_{N,b,\varepsilon}R^{-4}
\left(\|u\|_{L^2}^{2+4p/b}+\|u\|_{L^2}^{q}\right),
$$
which proves \eqref{eq:weighted-GN}.
\end{proof}

The lower-order linear commutators require a second interpolation estimate.

\begin{lemma}[Weighted gradient estimate]
\label{lem:weighted-gradient}
For every $\varepsilon>0$, there exists $C_{N,\varepsilon}>0$ such that, for every $u \in H^2(\R^N)$ and every $R>0$,
$$
R^{-2}\int_{\mathbb R^N}\omega_R^{1/2}|\nabla u|^2 dx
\le \varepsilon\mathcal D_R(u)
+ C_{N,\varepsilon}R^{-4}\|u\|_{L^2}^2.
$$
\end{lemma}

\begin{proof}
We use the notation introduced in the proof of Proposition~\ref{prop:weighted-GN}. Since $A_j\subset A_j^*$, the localized estimate \eqref{eq:localized-gradient} gives
$$
\|\nabla u\|_{L^2(A_j)}^2
\lesssim
m_j^{1/2}\|D^2u\|_{L^2(A_j^{**})}
+\rho_j^{-2}m_j.
$$
Moreover, since $\omega_R\simeq d_j$ on $A_j^{**}$,
$$
\|D^2u\|_{L^2(A_j^{**})}
\lesssim
d_j^{-1/2}D_j^{1/2}.
$$
Hence
$$
d_j^{1/2}\|\nabla u\|_{L^2(A_j)}^2
\lesssim
m_j^{1/2}D_j^{1/2}
+d_j^{1/2}\rho_j^{-2}m_j.
$$
Since $\omega_R^{1/2}\lesssim d_j^{1/2}$ on $A_j$, summing over $j$ yields
\begin{equation}
\label{eq:grad-est-prof-1}
\begin{aligned}
R^{-2}\int_{\mathbb R^N}
\omega_R^{1/2}|\nabla u|^2 dx
&\lesssim
R^{-2}\sum_jm_j^{1/2}D_j^{1/2}
+
R^{-2}\sum_jd_j^{1/2}\rho_j^{-2}m_j.
\end{aligned}
\end{equation}
For the first sum, Cauchy--Schwarz and the finite overlap of the enlarged annuli give
\begin{equation} \label{eq:grad-est-prof-2}
\begin{aligned}
R^{-2}\sum_jm_j^{1/2}D_j^{1/2}
&\le
R^{-2}\left(\sum_jm_j\right)^{1/2}
\left(\sum_jD_j\right)^{1/2} \
&\lesssim
R^{-2}\|u\|_{L^2}\mathcal D_R(u)^{1/2}.
\end{aligned}
\end{equation}
It remains to estimate the second sum in \eqref{eq:grad-est-prof-1}. We distinguish the inner and outer annuli.

If $\rho_j< R$, then
$$
d_j^{1/2}=\left(\frac{\rho_j}{R}\right)^2,
$$
and therefore
$$
R^{-2}d_j^{1/2}\rho_j^{-2}
=R^{-2}\left(\frac{\rho_j}{R}\right)^2\rho_j^{-2}
=R^{-4}.
$$
Thus
\begin{equation}\label{eq:grad-est-prof-3}
R^{-2}\sum_{\rho_j\le R} d_j^{1/2}\rho_j^{-2}m_j
\lesssim
R^{-4}\sum_{\rho_j\le R}m_j
\lesssim
R^{-4}\|u\|_{L^2}^2.
\end{equation}

If $\rho_j\ge R$, then $d_j=1$. Since $\rho_j=2^jR \ge R$ for $j\ge0$,
$$
R^{-2}d_j^{1/2}\rho_j^{-2}
=R^{-2}\rho_j^{-2}
=2^{-2j}R^{-4}.
$$
Consequently,
\begin{equation} \label{eq:grad-est-prof-4}
\begin{aligned}
R^{-2}
\sum_{\rho_j\ge R}
d_j^{1/2}\rho_j^{-2}m_j
&\le
R^{-4}\sum_{j\ge0}2^{-2j}m_j \
&\lesssim
R^{-4}\sum_{j\ge0}m_j
\lesssim
R^{-4}\|u\|_{L^2}^2.
\end{aligned}
\end{equation}
Combining \eqref{eq:grad-est-prof-1}--\eqref{eq:grad-est-prof-4}, we obtain
$$
R^{-2}\int_{\mathbb R^N}
\omega_R^{1/2}|\nabla u|^2 dx
\lesssim
R^{-2}\|u\|_{L^2}\mathcal D_R(u)^{1/2}
+R^{-4}\|u\|_{L^2}^2.
$$
Finally, Young's inequality gives
$$
R^{-2}\|u\|_{L^2}\mathcal D_R(u)^{1/2}
\le
\varepsilon\mathcal D_R(u)
+C_\varepsilon R^{-4}\|u\|_{L^2}^2.
$$
\enlargethispage{\baselineskip}
This proves the desired estimate.
\end{proof}

\section{The coercive localized virial estimate}
\label{sec:virial}

For a smooth real-valued weight $\phi$, define
\[
\mathcal M_\phi(u)=2\operatorname{Im}
\int_{\mathbb R^N} \overline u \nabla\phi\cdot\nabla u dx.
\]
For the equation
\[
i\partial_tu=\Delta^2u-|x|^{-b}|u|^pu,
\]
the localized biharmonic virial identity is
\[
\begin{aligned}
\frac{d}{dt}\mathcal M_\phi(u)={}&8\operatorname{Re}\sum_{i,j,k=1}^N
\int(\partial_{jk}\phi)(\partial_{ik}u)\overline{\partial_{ij}u} dx\\
&-4\operatorname{Re}\sum_{j,k=1}^N\int(\partial_{jk}\Delta\phi)(\partial_ju) \overline{\partial_ku} dx\\
&-2\int(\Delta^2\phi)|\nabla u|^2 dx+\int(\Delta^3\phi)|u|^2 dx\\
&-\frac{2p}{q}\int (\Delta\phi)|x|^{-b}|u|^q dx
+\frac4q \int \nabla\phi\cdot\nabla(|x|^{-b})|u|^q dx,
\end{aligned}
\]
where $q=p+2$. This is the standard biharmonic commutator identity, see e.g. \cite[Lemma 3.1]{BL} and \cite[Lemma 4.3]{Dinh} for the homogeneous setting, and \cite[Lemma 3.3]{Saanouni} for the corresponding inhomogeneous identity. Although \cite[Lemma 3.3]{Saanouni} is stated for $N\ge3$, the same commutator computation is valid for $N=1,2$.

For the exponential weight $\phi_R$ introduced in Section \ref{sec:multiplier}, we define the localized virial quantity
\[
\mathcal M_R(t):=\mathcal M_{\phi_R}(u(t))=2\operatorname{Im} \int_{\mathbb R^N} \overline{u(t,x)} \nabla\phi_R(x)\cdot\nabla u(t,x) dx.
\]
Since $\|\nabla\phi_R\|_{L^\infty}\lesssim R$, this is well defined for every $u\in H^2(\mathbb R^N)$.

\begin{proposition}[Coercive virial estimate]
\label{prop:coercive-virial}
Let $u$ be an $H^2$-solution to \eqref{eq:ibnls} defined on the maximal time interval $(-T_-,T_+)$. There exist constants $c_{N,b},C_{N,b}>0$ such that, for every $R>0$ and every $t\in (-T_-,T_+)$,
\[
\mathcal M_R'(t) \le 16E(u_0) -c_{N,b}\mathcal D_R(u(t))+C_{N,b}R^{-4}F_{b,p}(m),
\]
where
\[
F_{b,p}(m):=m+m^{q/2}+m^{1+2p/b}.
\]
\end{proposition}

\begin{proof}
For simplicity, write $u=u(t)$ and
\[
s=\frac{|x|}{R}, \quad h(s)=e^{-s^4}.
\]
Applying the localized biharmonic virial identity with $\phi=\phi_R$, we estimate its linear and nonlinear contributions separately.

\medskip
\noindent
\emph{Step 1. The principal Hessian term.}
By Lemma~\ref{lem:multiplier-geometry},
\[
D^2\phi_R(x)\le (1-\omega_R(x))I
\]
as quadratic forms. Hence, for every fixed $i$,
\[
\operatorname{Re} \sum_{j,k} (\partial_{jk}\phi_R) (\partial_{ik}u) \overline{\partial_{ij}u}
\le
(1-\omega_R) \sum_j|\partial_{ij}u|^2.
\]
Summing over $i$ and integrating, we obtain
\[
\begin{aligned}
8\operatorname{Re} \sum_{i,j,k} \int_{\mathbb R^N} (\partial_{jk}\phi_R) (\partial_{ik}u) \overline{\partial_{ij}u} dx
&\le
8\int_{\mathbb R^N}|D^2u|_F^2 dx - 8\mathcal D_R(u)\\
&= 8\|\Delta u\|_{L^2}^2 - 8\mathcal D_R(u),
\end{aligned}
\]
where the last equality follows from Plancherel's identity \eqref{eq:plancherel-iden}.

\medskip
\noindent
\emph{Step 2. The remaining linear terms.}
The lower-order linear contribution in the localized virial
identity is
\[
\begin{aligned}
L_R(u):={}&-4\operatorname{Re} \sum_{j,k} \int_{\mathbb R^N} (\partial_{jk}\Delta\phi_R) (\partial_j u) \overline{\partial_k u} dx\\
&\quad -2\int_{\mathbb R^N} (\Delta^2\phi_R)|\nabla u|^2 dx+\int_{\mathbb R^N}(\Delta^3\phi_R)|u|^2 dx.
\end{aligned}
\]
Notice in particular that the factor $\partial_{jk}\Delta\phi_R$ is paired with first derivatives of $u$, not with second derivatives.

By Lemma~\ref{lem:multiplier-derivatives},
\[
|\nabla^2\Delta\phi_R| + |\Delta^2\phi_R| \lesssim R^{-2}\omega_R^{1/2},
\quad
|\Delta^3\phi_R|\lesssim R^{-4}.
\]
Therefore
\[
|L_R(u)|\lesssim R^{-2} \int_{\mathbb R^N} \omega_R^{1/2}|\nabla u|^2 dx+ R^{-4}\|u\|_{L^2}^2.
\]
Using Lemma~\ref{lem:weighted-gradient}, for every $\varepsilon>0$,
\[
|L_R(u)| \le \varepsilon\mathcal D_R(u) + C_{N,\varepsilon}R^{-4}m.
\]

\medskip
\noindent
\emph{Step 3. The nonlinear contribution.}
The nonlinear part of the localized virial identity is
\[
-\frac{2p}{q} \int_{\mathbb R^N} (\Delta\phi_R)|x|^{-b}|u|^q dx + \frac4q \int_{\mathbb R^N} \nabla\phi_R\cdot\nabla(|x|^{-b})|u|^q dx.
\]
Since
\[
\Delta\phi_R=h(s)(N-4s^4)
\]
and
\[
\nabla\phi_R\cdot\nabla(|x|^{-b}) =-bh(s)|x|^{-b},
\]
this becomes
\[
-\frac2q \int_{\mathbb R^N} h(s) \bigl(Np+2b-4ps^4\bigr) |x|^{-b}|u|^q dx.
\]
Using the mass-critical relation
\[
Np+2b=8,
\]
we rewrite it as
\[
-\frac{16}{q} \int_{\mathbb R^N} |x|^{-b}|u|^q dx +\frac2q \int_{\mathbb R^N} \Theta_R(x)|x|^{-b}|u|^q dx,
\]
where
\[
\Theta_R(x):=8\omega_R(x)+4ps^4h(s).
\]
Since
\[
s^4e^{-s^4}\le 1-e^{-s^4}=\omega_R,
\]
we have
\[
0\le \Theta_R(x)\le (8+4p)\omega_R(x).
\]
Consequently, Proposition~\ref{prop:weighted-GN} yields, for every $\varepsilon>0$,
\[
\frac2q \int_{\mathbb R^N} \Theta_R(x)|x|^{-b}|u|^q dx \le \varepsilon\mathcal D_R(u) + C_{N,b,\varepsilon}R^{-4}
\left(m^{1+2p/b}+m^{q/2}\right).
\]

\medskip
\noindent
\emph{Step 4. Recovery of the conserved energy.}
Combining the principal kinetic contribution with the main
nonlinear term gives
\[
8\|\Delta u\|_{L^2}^2 - \frac{16}{q} \int_{\mathbb R^N} |x|^{-b}|u|^q dx=16E(u_0).
\]
Hence Steps 1--3 give
\[
\begin{aligned} \mathcal M_R'(t) \le{}& 16E(u_0) - (8-2\varepsilon)\mathcal D_R(u(t))\\
&\quad + C_{N,b,\varepsilon}R^{-4} \left( m+m^{q/2}+m^{1+2p/b}\right).
\end{aligned}
\]
Choosing $\varepsilon>0$ sufficiently small and absorbing it into the defect term, we obtain
\[
\mathcal M_R'(t) \le 16E(u_0) -c_{N,b}\mathcal D_R(u(t)) + C_{N,b}R^{-4}
\left(m+m^{q/2}+m^{1+2p/b}\right),
\]
which proves the proposition.
\end{proof}

\begin{remark}
The point of Proposition~\ref{prop:coercive-virial} is that the defect $\mathcal D_R$ is retained with a strictly negative sign. If this term were discarded, one would recover only a monotonicity estimate for $\mathcal M_R$, which by itself still allows a global grow-up scenario. The remainder of the argument couples the same defect back to the localized virial quantity.
\end{remark}

\section{Control of the virial by the defect}
\label{sec:virial-defect}

We next show that the same defect appearing in the coercive virial estimate also controls the localized virial itself.

\begin{proposition}[Virial--defect estimate]
\label{prop:virial-defect}
Let $N\geq 1$, $0<b<\min\{4,N/2\}$, and $u$ be an $H^2$-solution to \eqref{eq:ibnls} defined on the maximal time interval $(-T_-,T_+)$. There exists $C_N>0$ such that, for every $R>0$ and every $t\in (-T_-,T_+)$,
\[
|\mathcal M_R(t)| \le C_Nm + C_NR m^{3/4}\mathcal D_R(u(t))^{1/4}.
\]
\end{proposition}

\begin{proof}
Set
\[
a(x):=\nabla\phi_R(x)=xh(|x|/R)
\]
and
\[
X:=\|a\cdot\nabla u\|_{L^2}.
\]
By Cauchy--Schwarz,
\[
|\mathcal M_R(u)| \le 2m^{1/2}X.
\]
Using the summation convention,
\[
X^2 = \operatorname{Re} \int_{\mathbb R^N} a_ja_k(\partial_ju) \overline{\partial_ku} dx.
\]
Integrating by parts in the $j$ variable,
\[
\begin{aligned}
X^2={}& -\operatorname{Re} \int_{\mathbb R^N} u \partial_j(a_ja_k) \overline{\partial_ku} dx\\
&-\operatorname{Re} \int_{\mathbb R^N} u a_ja_k \overline{\partial_{jk}u} dx.
\end{aligned}
\]
A direct computation gives
\[
\partial_j(a_ja_k) = h(s)(N+1-8s^4)a_k.
\]
Since the scalar factor $h(s)(N+1-8s^4)$ is uniformly bounded, the first term is bounded by
\[
Cm^{1/2}X.
\]
For the second term, observe that
\[
|a(x)|^4=|x|^4e^{-4|x|^4/R^4}.
\]
Since
\[
s^4e^{-4s^4} \lesssim 1-e^{-s^4},
\]
we obtain
\[
|a(x)|^4 \lesssim R^4\omega_R(x).
\]
Therefore
\[
\begin{aligned}
\left|\int_{\mathbb R^N} u a_ja_k \overline{\partial_{jk}u} dx\right|
&\le m^{1/2} \left(\int_{\mathbb R^N} |a|^4|D^2u|_F^2 dx \right)^{1/2}\\
&\lesssim R^2m^{1/2}\mathcal D_R(u)^{1/2}.
\end{aligned}
\]
Thus
\[
X^2 \le Cm^{1/2}X + CR^2m^{1/2}\mathcal D_R(u)^{1/2}.
\]
Solving this quadratic inequality gives
\[
X \le Cm^{1/2} + CRm^{1/4}\mathcal D_R(u)^{1/4}.
\]
Hence
\[
|\mathcal M_R(u)| \le C_Nm + C_NR m^{3/4}\mathcal D_R(u)^{1/4}.
\]
\end{proof}

\section{Proof of the main theorem}
\label{sec:blowup}

\begin{proof}[Proof of Theorem~\ref{thm:main}]
Suppose now that
\[
E(u_0)=-e<0.
\]
Choose $R$ sufficiently large, depending only on $e$, $m$, $N$, and $b$, so that
\[
C_{N,b}R^{-4}F_{b,p}(m)\le8e.
\]
Then Proposition~\ref{prop:coercive-virial} yields
\[
\mathcal M_R'(t) \le -8e-c_{N,b}\mathcal D_R(u(t)).
\]
We first prove finite-time blow-up forward in time. Suppose, toward a contradiction, that
\[
T_+=\infty.
\]
Integrating the constant negative term gives
\[
\mathcal M_R(t) \le \mathcal M_R(0)-8et.
\]
Hence
\[
\mathcal M_R(t)\longrightarrow-\infty \quad \text{as }t\to\infty.
\]
Set
\[
y(t):=-\mathcal M_R(t).
\]
There exists $t_0<\infty$ such that
\[
y(t)\ge2C_Nm>0 \quad \text{for all }t\ge t_0,
\]
where $C_N$ is at least as large as the first constant in Proposition~\ref{prop:virial-defect}.

For $t\ge t_0$, Proposition~\ref{prop:virial-defect} gives
\[
\frac12y(t) \le C_NR m^{3/4} \mathcal D_R(u(t))^{1/4}.
\]
Consequently,
\[
\mathcal D_R(u(t)) \ge c_NR^{-4}m^{-3}y(t)^4, \quad t\ge t_0.
\]
On the other hand,
\[
y'(t)=-\mathcal M_R'(t)\ge c_{N,b}\mathcal D_R(u(t)).
\]
Therefore
\[
y'(t)\ge c_{N,b}R^{-4}m^{-3}y(t)^4,
\quad
t\ge t_0.
\]
Since $y(t)>0$,
\[
\frac{d}{dt}y(t)^{-3}=-3y(t)^{-4}y'(t)
\le-c_{N,b}R^{-4}m^{-3}.
\]
Integrating from $t_0$ to $t$ yields
\[
y(t)^{-3}\le y(t_0)^{-3}-c_{N,b}R^{-4}m^{-3}(t-t_0).
\]
Set 
\[
t_* := t_0 + \frac{y(t_0)^{-3}}{c_{N,b} R^{-4} m^{-3}}>t_0.
\]
Since $T_+=\infty$ and $u \in C([0,t_*];H^2(\R^N))$, $y(t_*)$ is finite and positive. However, the preceding estimate gives $y(t_*)^{-3} \leq 0$, a contradiction. Therefore, $T_+<\infty$. The blow-up alternative then gives
\[
\lim_{t\uparrow T_+} \|u(t)\|_{H^2} = \infty.
\]
For the backward statement, define
\[
\widetilde u(t,x):=\overline{u(-t,x)}.
\]
Then $\widetilde u$ solves the same equation and satisfies
\[
M(\widetilde u)=M(u), \quad E(\widetilde u)=E(u)<0.
\]
Applying the forward argument to $\widetilde u$ gives $T_-<\infty$ and
\[
\lim_{t\downarrow-T_-} \|u(t)\|_{H^2} = \infty.
\]
This completes the proof.
\end{proof}

\section{Strong instability of standing waves}
\label{sec:instability}

In this section, we prove Corollary~\ref{coro:instability}. Let $\omega\in\mathbb R$, and let $0\neq Q\in H^2(\mathbb R^N)$ be a standing-wave profile, that is, a solution of
\[
\Delta^2Q+\omega Q-|x|^{-b}|Q|^pQ=0.
\]
Then
\[
u(t,x)=e^{i\omega t}Q(x)
\]
is a standing-wave solution of \eqref{eq:ibnls}.

We first record the Pohozaev identities associated with the stationary equation.

\begin{lemma}[Pohozaev identities]
\label{lem:pohozaev-standing}
Let $N\ge 1$, $0<b<\min\{4,N/2\}$, and $p,q$ be as in \eqref{eq:def-pq}. Let $0\neq Q\in H^2(\mathbb R^N)$ solve
\[
\Delta^2Q+\omega Q-|x|^{-b}|Q|^pQ=0.
\]
Then
\[
\|\Delta Q\|_{L^2}^2+\omega\|Q\|_{L^2}^2-\int_{\mathbb R^N}|x|^{-b}|Q|^q dx=0
\]
and
\[
\frac{N-4}{2}\|\Delta Q\|_{L^2}^2+\frac N2\omega\|Q\|_{L^2}^2
-\frac{N-b}{q}\int_{\mathbb R^N}|x|^{-b}|Q|^q dx =0.
\]
In particular,
\[
E(Q)=0 \quad\text{and}\quad \omega>0.
\]
\end{lemma}

\begin{proof}
Multiplying the stationary equation by $\overline Q$, integrating over
$\mathbb R^N$, and taking real parts gives
\[
\|\Delta Q\|_{L^2}^2 +\omega\|Q\|_{L^2}^2 -\int_{\mathbb R^N}|x|^{-b}|Q|^q dx =0.
\]
Similarly, multiplying by $x\cdot\nabla\overline Q$, taking real parts, and integrating by parts yields
\[
\frac{N-4}{2}\|\Delta Q\|_{L^2}^2 +\frac N2\omega\|Q\|_{L^2}^2 -\frac{N-b}{q} \int_{\mathbb R^N}|x|^{-b}|Q|^q dx=0.
\]
The latter identity is justified for $H^2$ solutions by the standard cutoff argument, see e.g. \cite[Lemma~2.1]{BCGJ} for the corresponding biharmonic Pohozaev computation.

Set
\[
A:=\|\Delta Q\|_{L^2}^2,
\quad
B:=\|Q\|_{L^2}^2,
\quad
C:=\int_{\mathbb R^N}|x|^{-b}|Q|^q dx.
\]
Subtracting $\frac N2$ times the first identity from the second gives
\[
-2A+\left(\frac N2-\frac{N-b}{q}\right)C=0.
\]
Since
\[
q=\frac{2N+8-2b}{N},
\]
we have
\[
\frac N2-\frac{N-b}{q}=\frac4q.
\]
Therefore
\[
\frac12A=\frac1qC.
\]
This is precisely
\[
E(Q)=\frac12\|\Delta Q\|_{L^2}^2-\frac1q \int_{\mathbb R^N}|x|^{-b}|Q|^q dx=0.
\]
Returning to the first Pohozaev identity and using $C=qA/2$, we obtain
\[
\omega B=\left(\frac q2-1\right)A=\frac p2A.
\]
Equivalently,
\[
\omega\|Q\|_{L^2}^2=\frac p2\|\Delta Q\|_{L^2}^2=\frac{4-b}{N}\|\Delta Q\|_{L^2}^2.
\]
Since $Q\neq0$ and $0<b<4$, the right-hand side is strictly positive. Hence
\[
\omega>0.
\]
\end{proof}

We are now in a position to prove the strong instability result.

\begin{proof}[Proof of Corollary~\ref{coro:instability}]
Let $\varepsilon>0$. For $c>1$, define
\[
Q_c:=cQ.
\]
By Lemma~\ref{lem:pohozaev-standing},
\[
E(Q)=0,
\]
and hence
\[
\frac12\|\Delta Q\|_{L^2}^2=\frac1q \int_{\mathbb R^N}|x|^{-b}|Q|^q dx.
\]
It follows that
\[
\begin{aligned}
E(Q_c)&=\frac{c^2}{2}\|\Delta Q\|_{L^2}^2-\frac{c^q}{q}\int_{\mathbb R^N}|x|^{-b}|Q|^q dx\\
&=\frac12\|\Delta Q\|_{L^2}^2\left(c^2-c^q\right).
\end{aligned}
\]
Since $p>0$, we have $q=p+2>2$. Therefore, for every $c>1$,
\[
c^q>c^2,
\]
and consequently
\[
E(Q_c)<0.
\]

On the other hand,
\[
\|Q_c-Q\|_{H^2}=(c-1)\|Q\|_{H^2}\longrightarrow0
\quad
\text{as }c\downarrow1.
\]
Thus we may choose $c>1$ sufficiently close to $1$ so that
\[
\|Q_c-Q\|_{H^2}<\varepsilon.
\]
Applying Theorem~\ref{thm:main} with initial datum $u_0=Q_c$, we conclude that the corresponding solution blows up in finite positive time. Hence the standing wave $e^{i\omega t}Q$ is strongly unstable by finite-time blow-up.
\end{proof}

\begin{remark}
The argument uses only the fact that $Q$ is a nontrivial solution of the stationary equation. In particular, no ground-state characterization, minimizing property, positivity, or radial symmetry is required. Thus Corollary~\ref{coro:instability} applies to every nontrivial standing wave.
\end{remark}

\end{document}